\documentclass[a4paper,12pt]{article}
\usepackage{amssymb,amsthm,amsmath,amsfonts}
\usepackage{graphicx}
\usepackage{cite}
\usepackage{color}

\newtheorem{theorem}{Theorem}[section]

\newtheorem{corollary}[theorem]{Corollary}
\newtheorem{proposition}[theorem]{Proposition}

\theoremstyle{definition}
\newtheorem{definition}[theorem]{Definition}

\theoremstyle{remark}
\newtheorem{remark}{Remark}

\begin{document}

\title{A quaternionic fractional Borel-Pompeiu type formula}
\small{
\author
{Jos\'e Oscar Gonz\'alez-Cervantes$^{(1)}$ and Juan Bory-Reyes$^{(2)\footnote{corresponding author}}$}
\vskip 1truecm
\date{\small $^{(1)}$ Departamento de Matem\'aticas, ESFM-Instituto Polit\'ecnico Nacional. 07338, Ciudad M\'exico, M\'exico\\ Email: jogc200678@gmail.com\\$^{(2)}$ {SEPI, ESIME-Zacatenco-Instituto Polit\'ecnico Nacional. 07338, Ciudad M\'exico, M\'exico}\\Email: juanboryreyes@yahoo.com
}

\maketitle

\begin{abstract}
In theoretical setting, associated with a fractional $\psi-$Fueter operator that depends on an additional vector of complex parameters with fractional real parts, this paper establishes a fractional analogue of Borel-Pompeiu formula as a first step to develop a fractional $\psi-$hyperholomorphic function theory and the related operator calculus.
\end{abstract}


\section{Introduction} 
Fractional calculus is a theory allowing integrals and derivatives of arbitrary real or complex order. The interest in the subject has been growing continuously during the last few decades because of numerous applications in diverse fields of science and engineering, see \cite{BMBD, GM, KST, OS, O, P, MR, T, SKM, SDKSKB}. Fractional integrals and derivatives can be traced back to the genesis of differential calculus when in 1695 G. Leibnitz mentioned a derivative of order $\displaystyle\frac{1}{2}$. However a rigorous investigation was first carried out by J. Liouville in a series of papers from 1832-1837, where he defined the first outcast of an operator of fractional integration, see for instance \cite{L}. Later investigations and further developments by among others B. Riemann \cite{R} led to the construction of the integral-based Riemann-Liouville fractional integral operator (see \cite{CC}), which has been a valuable cornerstone in fractional calculus ever since. For a brief history and exposition of the fundamental theory of fractional calculus we refer the reader to \cite{Ro}.

A framework for a fractional Euclidian Clifford analysis represents a very recent topic of research, see \cite{CDOP, DM, FRV, KV, PBBB, V} and the references given there. In particular the interest for considering fractional Dirac Operator is devoted in \cite{Ba, Be, FV1, FV2}.

Nowadays, quaternionic analysis is regarded as a broadly accepted branch of classical analysis offering a successful generalization of complex holomorphic function theory, the most renowned examples are Sudbery's paper\cite{sudbery} and the books \cite{GS1, GS2, KS, K}. It relies heavily on results on functions defined on domains in $\mathbb R^4$ with values in the skew field of real quaternions $\mathbb H$ associated to a generalized Cauchy-Riemann operator (the so-called $\psi-$Fueter operator) by using a general orthonormal basis in $\mathbb R^4$ (to be named structural set) $\psi$ of $\mathbb H^4$, see, e.g., \cite{MS, Na, No1, No2}. This theory is centered around the concept of $\psi-$hyperholomorphic functions, see \cite{MS, S1, S2, shapiro1, shapiro2}.

The original contribution of this paper is to give a re-contextualization of the $\psi-$hyperholomorphic function theory and the related operator calculus in a fractional setting, based on the modification of the $\psi-$Fueter operator in the Riemann-Liouville sense that depends on an additional vector of complex parameters with fractional real parts. In addition to the general extent of our proposed research, the important fact that the introduced fractional $\psi-$Fueter operator factorizes a fractional $\psi-$Laplacian is also shown.

The structure of the papers is as follows. After this brief introduction, in the preliminary section we recall some basic facts about a quaternionic analysis associated to a structural set $\psi$, such as the $\psi$-Fueter operator, Stokes and the Borel-Pompieu formulas, as well as the notions of fractional Riemann-Liouville integral and derivative. Sec. 3 deals with a quaternionc version of the fractional derivative of the Riemann-Loiuville associated to a structural set. Stokes and Borel-Pompieu type formulas are proved.

\section{Preliminaries} 
Below we give basic definitions and facts on the fractional calculus and quaternionic analysis. These notions will be used throughout the whole paper.
\subsection{Standard definition of and results on Riemann-Liouville fractional integro-differential operators}
There are different definitions of fractional derivatives. One of the most popular (even though it has disadvantages for applications to real world problems) is the Riemann–Liouville derivative (see, e.g., \cite{MR}). For completeness, we recall the key definitions and results on Riemann-Liouville fractional integro-differential operators

Given $\alpha\in \mathbb C$ with $\Re \alpha> 0$, let us recall that the Riemann-Liouville integrals of order $\alpha$ of a  $f  \in L^1([a, b], \mathbb R)$, with  $-\infty <a  < b< \infty$, on the left and on the right, are defined  by 
$$({\bf I}_{a^+}^{\alpha} f)(x) := \frac{1}{\Gamma(\alpha)} \int_a^x \frac{f(\tau)}{(x-\tau)^{1-\alpha}} d\tau, \quad \textrm{with}  \quad x > a$$
and
$$({\bf I}_{b^-}^{\alpha} f)(x) := \frac{1}{\Gamma(\alpha)} \int_x^b \frac{f(\tau)}{(\tau-x)^{1-\alpha}} d\tau, \quad \textrm{with}  \quad x < b,$$
respectively. 

What is more, let $n=[\Re \alpha]+1$, where $[\cdot]$ means the integer part of $\cdot$ and $f\in AC^n([a, b], \mathbb R)$; i.e., the class of functions $f$ which are continuously differentiable on the segment $[a, b]$ up to the order $n-1$ and $f^{(n-1)}$ is supposed to be absolutely continuous on $[a, b]$. The fractional derivatives in the Riemann-Liouville sense, on the left and on the right, are defined by 
\begin{align}\label{FracDer} 
(D _{a^+}^{\alpha} f)(x):= \frac{d}{dx^n} \left[ ({\bf I}_{a^+}^{n-\alpha} f)(x)\right]
\end{align}
and
\begin{align} \label{FracDer1}
(D _{b^-}^{\alpha} f)(x):= (-1)^n\frac{d}{dx^n}\left[({\bf I}_{b^-}^{n-\alpha}f)(x)\right] 
\end{align}
respectively. It is worth noting that the derivatives in (\ref{FracDer}), (\ref{FracDer1}) exist for $f\in AC^n([a, b], \mathbb R)$. Fractional Riemann-Liouville integral and derivative are linear operators.

Fundamental theorem for Riemann-Liouville fractional calculus \cite{CC} shows that 
\begin{align}\label{FundTheorem}
(D_{a^+}^{\alpha} {\bf I}_{a^+}^{\alpha}f)(x)=f(x) \quad  \textrm{and} \quad (D _{b^-}^{\alpha}  {\bf I}_{b^-}^{\alpha} f)(x) = f(x).
\end{align}

Let us mention an important property of the fractional Riemann-Liouville integral and derivative, see \cite[pag. 23]{BKT}, \cite[pag. 1835]{VTRMB}.
\begin{proposition}
\begin{equation}\label{cte}
( D _{a^+}^{\alpha} 1)(x)=\frac{(x-a)^{-\alpha}}{\Gamma[1-\alpha]}  , \quad \forall   x\in [a, b].
\end{equation}
\end{proposition}
\subsection{Rudiments of quaternionic analysis}
Consider the skew field of real quaternions $\mathbb H$ with its basic elements $1, {\bf i},  {\bf j},  {\bf k}$. Thus any element $x$ from $\mathbb H$ is of the form $x=x_0+x_{1} {\bf i}+x_{2} {\bf j}+x_{3} {\bf k}$, $x_{k}\in \mathbb R, k= 0,1,2,3$. The basic elements define arithmetic rules in $\mathbb H$: by definition $ {\bf i}^{2}= {\bf j}^{2}= {\bf k}^{2}=-1$, $ {\bf i}\, {\bf j}=- {\bf j}\, {\bf i}= {\bf k};  {\bf j}\, {\bf k}=- {\bf j}\, {\bf k}= {\bf i}$ and $ {\bf k}\, {\bf i}=- {\bf k}\, {\bf i}= {\bf j}$. For $x\in \mathbb H$ we define the mapping of quaternionic conjugation: $x\rightarrow {\overline x}:=x_0-x_{1}{\bf i}-x_{2}{\bf j}-x_{3}{\bf k}$. In this way it is easy seen that $x\,{\overline x}={\overline x}\,x=x^{2}_{0}+x^{2}_{1}+x^{2}_{2}+x^{2}_{3}$. Note that $\overline {qx}={\overline x}\,\,{\overline q}$ for $q,x\in \mathbb H$. 

The quaternionic scalar product  of $q, x\in\mathbb H$ is given by 
$$\langle q, x\rangle:=\frac{1}{2}(\bar q x + \bar x q) = \frac{1}{2}(q \bar x + x  \bar q).$$

A set of quaternions $\psi=\{\psi_0, \psi_1,\psi_2,\psi_2\}$ is called structural set if $\langle \psi_k, \psi_s\rangle =\delta_{k,s} $, 
for  {$k, s=0,1,2,3$ and} any quaternion $x$  can be rewritten as $x_{\psi} := \sum_{k=0}^3 x_k\psi_k$, where $x_k\in\mathbb R$ for all $k$. Given $q, x\in \mathbb H$ we follow the notation used in {\cite{shapiro1}} to write
$$\langle q, x \rangle_{\psi}=\sum_{k=0}^3 q_k x_k,$$ 
where $q_k, x_k\in \mathbb R$ for all $k$.

Let $\psi$ an structural set. From now on, we will use the mapping
\begin{equation} \label{mapping}
\sum_{k=0}^3 x_k\psi_k \rightarrow (x_0,x_1,x_2,x_3).
\end{equation}
in essential way.

We have to say something about the set of complex quaternions, which are given by  
$$\mathbb H(\mathbb C) =\{q=q_1+  \textsf{i} \ q_2 \ \mid \ q_1,q_2 \in \mathbb H\},$$
where $\textsf{ i}$   is the imaginary unit of $\mathbb C$.
The main difference to the real quaternions is that not all non-zero elements are invertible. There are so-called zero-divisors.

Let us recall that $\mathbb H$ is embedded in $\mathbb H(\mathbb C)$ as follows:
$$\mathbb H =\{q=q_1+   \textsf{i}  \ q_2 \in \mathbb H(\mathbb C)  \ \mid \  q_1,q_2 \in \mathbb H \ \ \textrm{and} \ \ q_2=0\}.$$
The elements of $\mathbb H$ are written in terms of the structural set $\psi$ hence those of $\mathbb H(\mathbb C)$ can be written as $q=\sum_{k=0 }^3 \psi_k q_k,$ where $q_k\in \mathbb C$.
 
Functions $f$ defined in a bounded domain $\Omega\subset\mathbb H\cong \mathbb R^4$ with value in $\mathbb H$ are considered. They may be written as: $f=\sum_{k=0}^3 f_k \psi_k$, where $f_k, k= 0,1,2,3,$ are $\mathbb R$-valued functions in $\Omega$. Properties as continuity, differentiability, integrability and so on, which as ascribed to $f$ have to be posed by all components $f_k$. We will follow standard notation, for example $C^{1}(\Omega, \mathbb H)$ denotes the set of continuously differentiable $\mathbb H$-valued functions defined in $\Omega$.  

The left- and the right-$\psi$-Fueter operators are defined by   
${}^{{\psi}}\mathcal D[f] := \sum_{k=0}^3 \psi_k \partial_k f$ and ${}^{{\psi}}\mathcal  D_r[f] :=  \sum_{k=0}^3 \partial_k f \psi_k$, for all $f \in C^1(\Omega,\mathbb H)$, respectively, where $\partial_k f =\displaystyle \frac{\partial f}{\partial x_k}$ for all $k$, {see \cite{shapiro1, shapiro2}. 

Particularly,  {if $\partial \Omega$ is a 3-dimensional smooth surface then the Borel-Pompieu formula shows that
\begin{align}\label{BorelHyp}  &  \int_{\partial \Omega}(K_{\psi}(\tau-x)\sigma_{\tau}^{\psi} f(\tau)  +  g(\tau)   \sigma_{\tau}^{\psi} K_{\psi}(\tau-x) ) \nonumber  \\ 
&  - 
\int_{\Omega} (K_{\psi} (y-x) {}^{\psi}\mathcal D [f] (y) + {}^{{\psi}}\mathcal  D_r [g] (y) K_{\psi} (y-x)
     )dy   \nonumber \\
		=  &  \left\{ \begin{array}{ll}  f(x) + g(x) , &  x\in \Omega,  \\ 0 , &  x\in \mathbb H\setminus\overline{\Omega}.                     
\end{array} \right. 
\end{align} 
{Differential and integral versions of Stokes' formulas for the $\psi$-hyperholomorphic functions theory are given by 
\begin{align}\label{StokesHyp} d(g\sigma^{{\psi} }_x f) = & \left(g \ {}^{{\psi}}\mathcal  D[f]+ \  {}^{{\psi}}\mathcal D_r[g] f\right)dx,\\
		\int_{\partial \Omega} g\sigma^\psi_x f =  &   \int_{\Omega } \left( g {}^\psi \mathcal  D[f] + {}^{{\psi}}\mathcal  D_r[g] f\right)dx,
\end{align}
for all $f,g \in C^1(\overline{\Omega}, \mathbb H)$},  {see \cite{shapiro1, shapiro2, sudbery}}. Here, $d$ stands for the exterior differentiation operator, $dx$ denotes the differential form of the 4-dimensional volume in $\mathbb R^4$ and  
$$\sigma^{{\psi} }_{x}:=-sgn\psi \left( \sum_{k=0}^3 (-1)^k \psi_k d\hat{x}_k\right)$$ 
is the quaternionic differential form of the 3-dimensional volume in $\mathbb R^4$ according  to $\psi$, where $d\hat{x}_k  = dx_0 \wedge dx_1\wedge dx_2  \wedge  dx_3 $ omitting factor $dx_k$.  {In addition,} $sgn\psi$ is $1$, or $-1$,  if  $\psi$ and  $\psi_{std}:=\{1, {\bf i}, {\bf j}, {\bf k}\}$ have the same orientation, or not, respectively.  {Note that}, $|\sigma^{{\psi} }_{x}| = dS_3$  is the differential form of the 3-dimensional volume in $\mathbb R^4$ and write $\sigma_x=\sigma^{{\psi_{std}} }_{x}$. Let us recall that the $\psi$-hyperholomorphic Cauchy Kernel is given by 
 \[ K_{\psi}(\tau- x)=\frac{1}{2\pi^2} \frac{ \overline{\tau_{\psi} - x_{\psi}}}{|\tau_{\psi} - x_{\psi}|^4},\]
and the integral operator 
$${}^{\psi}\mathcal T[f](x) = \int_{\Omega} K_{\psi} (y-x) f  (y) dy$$ 
defined for all $f\in L_2(\Omega,\mathbb H)\cup C(\Omega,\mathbb H)$ satisfies 
\begin{align}\label{FueterInv}{}^{\psi}\mathcal D \circ{}^{\psi}\mathcal T[f]=f, \ \ \forall f\in L_2(\Omega,\mathbb H)\cup C(\Omega,\mathbb H).
\end{align} 
This can be found in \cite{MS, S1, S2, shapiro1, shapiro2}.  

\section{Main results}
We shall consider vector parameters $\vec{\alpha} = (\alpha_0, \alpha_1,\alpha_2,\alpha_3) \in\mathbb C^4$, requiring the values of $\Re\alpha_{\ell}$, for $\ell=0,1,2,3$, are varied between any two consecutive integer values, that is, $n=[\Re \alpha_{\ell}]+1$, for $n\in \mathbb N$.

Firstly, let us start with the case  $n=1, i.e., 0 < \Re\alpha_{\ell} < 1$ for $\ell=0,1,2,3$ and finally the general case will be presented.  
\subsection{Fractional $\psi$-Fueter operator of order $\vec{\alpha}$}
\begin{definition}
Let $a=\sum_{k=0}^3\psi_k a_k, b=\sum_{k=0}^3\psi_k b_k \in \mathbb H$ such that $a_k< b_k$ for all $k$. Write 
\begin{align*}  {J_a^b }:= &  \{  \sum_{k=0}^3\psi_k x_k \in \mathbb H \ \mid \ a_k< x_k < b_k, \  \ k=0,1,2,3\} \\
 = & (a_0,b_0) \times (a_1,b_1) \times (a_2,b_2)  \times (a_3,b_3) ,
\end{align*}
and define $m(J_a^b):=(b_0-a_0) (b_1-a_1)(b_2-a_2)(b_3-a_3)$. 
\end{definition}
In what follows, with notation ${J_a^b}$ we assume $a_k < b_k$ for all $k$.
\begin{remark} 
Set {$\vec{\alpha} = (\alpha_0, \alpha_1,\alpha_2,\alpha_3) \in\mathbb C^4$} and $f=\sum_{i=0}^3\psi_i f_i\in AC^1(J_a^b,\mathbb H)$; i.e., $f_0, \ f_1,\ f_2, \ f_3$, the real components of $f$, belongs to $AC^1(\Omega,\mathbb R).$ The mapping $x_j \mapsto f_i(q_0,\dots,x_j,\dots, q_3)$ belongs to $AC^1((a_i, b_i), \mathbb R)$ for each $q\in J_a^b$ and all $i, j=0,1,2,3$.

Now, given $q, x\in  J_a^b $ and $i, j=0,\dots, 3$, the fractional integral of Riemann-Liouville of order { $\alpha_j$} of the mapping 
$x_j \mapsto f_i(q_0,\dots,x_j,\dots, q_3 )$ is defined by: 
$$   ({\bf I}_{a_j^+}^{      {  \alpha_j  }   } f_i)(q_0, \dots, x_j, \dots, q_3)  = \frac{1}{\Gamma(
     {  \alpha_j  }  )} \int_{a_j}^{x_j} \frac{f_i  
(q_0, \dots, \tau_j, \dots, q_3)
}{(x_j-\tau_j)^{1-          {   \alpha_j }  }  } d\tau_j.$$

By the above, as $\displaystyle f=\sum_{i=0}^3 \psi_i f_i$ it follows that 
\begin{align*}  ({\bf I}_{a_j^+}^{    {  \alpha_{j}  } } f )(q_0, \dots, x_j, \dots, q_3) := &
\frac{1}{\Gamma(    {  \alpha_j }   )} \int_{a_j}^{x_j} \frac{f  
(q_0, \dots, \tau_j, \dots, q_3)
}{(x_j-\tau_j)^{1-    {  \alpha_j  }   }}  d\tau_j  \\
= &
 \sum_{i=0}^3 \psi_i\frac{1}{\Gamma(   { \alpha_j }   )} \int_{a_j}^{x_j} \frac{f_i  
(q_0, \dots, \tau_j, \dots, q_3)
}{(x_j-\tau_j)^{1-  { \alpha_j } }} d\tau_j \\
= & 
\sum_{i=0}^3 \psi_i ({\bf I}_{{a_j}^+}^{ { \alpha_j }   } f_i)(q_0, \dots, x_j, \dots, q_3)  .
\end{align*}     
for every $f \in AC^1(J_a^b,\mathbb H)$ and $q, x \in J_a^b$.

What is more, the fractional derivative in the Riemann-Liouville sense of the mapping $x_j\mapsto f(q_0, \dots, x_j, \dots, q_3)$  of order $\alpha_j$ is given by
\begin{align*} 
D _{a_j^+}^{      {  \alpha_j  }   } f(q_0, \dots ,  x_j ,\dots ,  q_3) = &
\frac{\partial   }{\partial x_j}
 \sum_{i=0}^3 \psi_i ({\bf I}_{a_j^+}^{      {  \alpha_j  }    } f_i)(q_0, \dots, x_j, \dots, q_3) \\
 = &
 \frac{\partial   }{\partial x_j}
 \frac{1}{\Gamma(      {  \alpha_j  }   )} \int_{a_j}^{x_j} \frac{f  
(q_0, \dots, \tau_j, \dots, q_3)
}{(x_j-\tau_j)^{1-      {  \alpha_j  }      }} d\tau_j .
\end{align*}
Note that $({\bf I}_{a_j^+}^{{\alpha_j}}f)$ and $D _{a_j^+}^{{\alpha_j}}f$ are $\mathbb H(\mathbb C)$-valued functions for every $j$. In a similar way we can introduce  $({\bf I}_{b_j^-}^{{\alpha_j}}f)$ and $D _{b_j^-}^{{\alpha_j}}f$. 
\end{remark}

\begin{definition} Let $f \in AC^1(J_a^b,\mathbb H)$, and $\vec{\alpha} = (\alpha_0, \alpha_1,\alpha_2,\alpha_3) \in\mathbb C^4$ such that $0< \Re \alpha_{\ell}<1$ for $\ell=0,1,2,3$. The fractional $\psi$-Fueter operator of order $\vec{\alpha}$ is defined to be
\begin{align*} 
{}^{\psi}\mathfrak D_a^{\vec{\alpha}}[f] (q,x):= & \sum_{j=0}^3 \psi_j( D _{a_j^+}^{{\alpha_j}}f)(q_0, \dots, x_j , \dots,  q_3)    \\
= &
\sum_{j=0}^3 \psi_j \frac{\partial}{\partial x_j} \frac{1}{\Gamma ({\alpha_j})}
  \int_{a_j}^{x_j} \frac{f  
(q_0, \dots, \tau_j, \dots, q_3)}{(x_j-\tau_j)^{1-{\alpha_j}}}d\tau_j , 
 \end{align*}
for $q, x\in J_a^b$. Note that $q$ is considered a fixed point since the integration and derivation variables are the real components of $x$. Moreover, ${}^{\psi}\mathfrak D_a^{\vec{\alpha}}[f](q,\cdot)$ is a $\mathbb H(\mathbb C)$-valued function.

Particularly, ${}^{\psi}\mathfrak D_a^{\vec{\alpha}}[f] (q,x) \mid_{x=q}$ can be considered  as ${}^{\psi}\mathfrak D_a^{\vec{\alpha}}[f] $ at point $q$. Then  denote ${}^{\psi}\mathfrak D_a^{\vec{\alpha}}[f] (q,x) \mid_{x=q} ={}^{\psi}\mathfrak D_a^{\vec{\alpha}}[f] (q) $.

On the other hand, given $f \in AC^1(J_a^b,\mathbb H)$ define 
$${}^{\psi} \mathfrak I_a^{\vec{\alpha}}[f](q,x) = \sum_{j=0}^3\frac{1}{2\Gamma(\alpha_j)} \int_{a_j}^{x_j}\frac{\bar \psi_j 
f(q_0,\dots, \tau_j, \dots, q_3)  + \overline{f(q_0,\dots, \tau_j, \dots, q_3)} \psi_j}{ (x_j- \tau_j)^{1-\alpha_j}} 
d\tau_j  $$  
 and 
	\begin{align*}       {}^{\psi}\mathcal I_a^x [f] (q,x,\vec{\alpha})    
	 =   \int_{J_a^x } \frac{f  
(\tau_0, q_1, \dots, q_3) \frac{(x_0-\tau_0)^{ \alpha_0}}{\Gamma(\alpha_0)}    + \dots+    f  
(q_0, \dots,  q_2, \tau_3)  \frac{(x_3-\tau_3)^{ \alpha_3}  }{\Gamma(\alpha_3)} 
   }{m(J_a^x) }  d\mu_{
\tau},
\end{align*}  
where  $\tau = \sum_{k=0}^3 \psi_k \tau_k$ and $d\mu_{\tau}$ is the differential of volume. 
\end{definition}
\begin{remark}
Throughout the paper we shall be mainly interested in the study of null-solutions of the operator ${}^{\psi}\mathfrak D_a^{\vec{\alpha}}$ and Riemann-Liouville integrals ${}^{\psi} \mathfrak I_a^{\vec{\alpha}}$ and ${}^{\psi}\mathcal I_a^x$. But similar analyzes have arisen for operators ${}^{\psi}\mathfrak D_b^{\vec{\alpha}}$, ${}^{\psi} \mathfrak I_b^{\vec{\alpha}}$ and ${}^{\psi}\mathcal I_x^b$ introduced with $({\bf I}_{a_j^+}^{{\alpha_j}})$ and $D _{a_j^+}^{{\alpha_j}}$ replaced by $({\bf I}_{b_j^-}^{{\alpha_j}})$ and $D _{b_j^-}^{{\alpha_j}}$ and the corresponding right chance of integration in ${}^{\psi} \mathfrak I_a^{\vec{\alpha}}$ and ${}^{\psi}\mathcal I_a^x$. 
\end{remark}

\begin{proposition}\label{propFRACD}
Consider $f \in AC^1(J_a^b,\mathbb H)$, and $ \vec{\alpha} = (\alpha_0, \alpha_1,\alpha_2,\alpha_3) \in\mathbb C^4$ with $0< \Re \alpha_{\ell} <1$ for $\ell=0,1,2,3$.   Then 
		\begin{enumerate} 
		\item $\displaystyle 
{}^{\psi}\mathfrak D_a^{\vec{\alpha}}[f](q,x) =  {}^{\psi}\mathcal D_x \circ {}^{\psi}\mathcal I_a^x [f](q, x,\vec{\alpha}).$

\item $\displaystyle 
{}^{\psi}\mathfrak D_a^{\vec{\alpha}} \circ  {}^{\psi} \mathfrak I_a^{\vec{\alpha}}[f](q,x) = \sum_{j=0}^3 \psi_j f_j(q_0,\dots, x_j, \dots, q_3)$     
for $f\in AC^{1}({{J_a^b}},\mathbb H)$, where $x= \sum_{j=0}^3 \psi_j x_j$ and $q= \sum_{j=0}^3 \psi_j q_j\in J_a^b$.
 
Particularly, $$\displaystyle 
  {}^{\psi}\mathfrak D_a^{\vec{\alpha}} \circ  {}^{\psi} \mathfrak I_a^{\vec{\alpha}}[f](q,x)\mid_{x=q}= f(q).$$

\item $\displaystyle {}^{\bar \psi}\mathcal D_x \circ {}^{\psi}\mathfrak D_a^{\vec{\alpha}}[f](q,x) = \Delta_{\mathbb R^4}\circ {}^{\psi}\mathcal I_a^x [f](q, x,\vec{\alpha})$,
where $\Delta_{\mathbb R^4}$ denotes the Laplacian in $\mathbb R^4$ according to the real components of $x$.

\item  If  the mapping $x\to \mathcal I_a^{x} [f](q, x,\vec{\alpha})$ belongs to $C^2(J_a^b, \mathbb H)$  for all $q$ and set 
$\vec{\beta} = (\beta_0, \beta_1, \beta_2, \beta_3) \in\mathbb C^4$ with $0< \Re\beta_{\ell} <1$ for $\ell=0,1,2,3$  then we have  

\begin{align*} 
{}^{\psi}\mathfrak D_a^{\vec{\alpha}} \circ  {}^{\psi}\mathfrak D_a^{\vec{\beta}}[f] (q,x) = & \sum_{j=0}^3 \psi_j^2 D _{a_j^+}^{{\alpha_j + \beta_j}} f(q_0, \dots, x_j , \dots,  q_3)   
\end{align*} 
and 
\begin{align*} 
{}^{\bar\psi}\mathfrak D_a^{\vec{\alpha}} \circ  {}^{\psi}\mathfrak D_a^{\vec{\beta}}[f](q,x) = & \sum_{j=0}^3 D _{a_j^+}^{{\alpha_j + \beta_j}}f(q_0, \dots, x_j ,\dots, q_3).
\end{align*}
Note that, for $\vec{\alpha}=\vec{\beta}$, the above formula drawn the fact that the fractional $\psi$-Fueter operator of order $\displaystyle\frac{1+\vec{\alpha}}{2}$ factorizes a fractional $\psi$-Laplace operator defined by ${}^{\psi}\Delta_a^{\vec{\alpha}}:=\sum_{j=0}^3 D _{a_j^+}^{{1+\alpha_j}}$.
\end{enumerate}
\end{proposition}
\begin{proof} 
\begin{enumerate}
\item The first identity is a consequence from the following: 
\begin{align*} 
&    \sum_{j=0}^3  ({\bf I}_{{a_j}^+}^{\alpha_j} f )(q_0, \dots, x_j, \dots, q_3)  
 \\
 = & \sum_{j=0}^3  
    {  \frac{1}{\Gamma(\alpha_j)}  } 
  \int_{a_j}^{x_j} \frac{f  
(q_0, \dots, \tau_j, \dots, q_3)
}{(x_j-\tau_j)^{1-   
    {\alpha_j} 
  }} d\tau_j  \\
 =  &      {  \frac{1}{\Gamma(\alpha_0)}  }   \int_{a_0}^{x_0} \frac{f  
(\tau_0, q_1, \dots, q_3)
}{(x_0-\tau_0)^{1-
    {\alpha_0}
}} d\tau_0  + \cdots+ 
    {  \frac{1}{\Gamma(\alpha_3)}  } 
\int_{a_3}^{x_3} \frac{f  
(q_0, \dots, q_2,  \tau_3)
}{(x_3-\tau_3)^{1-     {\alpha_3}    }} d\tau_3 
\\
 =  &       {  \frac{1}{\Gamma(\alpha_0)}  }  \int_{a_0}^{x_0} \frac{f  
(\tau_0, q_1, \dots, q_3)
 (x_1-a_1)\cdots (x_3-a_3)}{(x_0-\tau_0)^{1-    {\alpha_0} } (x_1-a_1)\cdots (x_3-a_3) } d\tau_0  +
 \cdots \\ 
 & \hspace{1cm} + 
    {  \frac{1}{\Gamma(\alpha_3)}  } 
\int_{a_3}^{x_3} \frac{f  
(q_0, \dots,  q_2, \tau_3)(x_0-a_0)\cdots (x_2-a_2)
}{(x_3-\tau_3)^{1-         {\alpha_3}          } (x_0-a_0)\cdots (x_2-a_2) } d\tau_3 \\
 =  &\int_{a_3}^{x_3} \int_{a_2}^{x_2}   \int_{a_1}^{x_1} \int_{a_0}^{x_0}   
  \frac{f  
(\tau_0, q_1, \dots, q_3)          {  \frac{(x_0-\tau_0)^{  {\alpha_0}  } }{\Gamma(\alpha_0)}  }    + \dots+    f  
(q_0, \dots, q_2,  \tau_3)     {  \frac{(x_3-\tau_3)^{  {\alpha_3}  } }{\Gamma(\alpha_3)}  }  
   }{(x_0-a_0)  (x_1-a_1)(x_2-a_2) (x_3-a_3) } \\ 
	&\\
		&  \hspace{3cm}d\tau_0  d\tau_1 d\tau_2 d\tau_3  
 \end{align*}
From Fubbini's Theorem one has that
\begin{align*} 
& \sum_{j=0}^3      {  \frac{1}{\Gamma(\alpha_j)}  }  
  \int_{a_j}^{x_j} \frac{f  
(q_0, \dots, \tau_j, \dots, q_3)
}{(x_j-\tau_j)^{1-      { \alpha_j  }  }}  d\tau_j  \\
 =  &\int_{J_a^x }   
  \frac{f  
(\tau_0, q_1, \dots, q_3)      {  \frac{(x_0-\tau_0)^{  {\alpha_0}  } }{\Gamma(\alpha_0)}  }   + \dots+    f  
(q_0, \dots,  q_2, \tau_3)      {  \frac{(x_3-\tau_3)^{  {\alpha_3}  } }{\Gamma(\alpha_3)}  } 
   }{m(J_a^x) }  d\mu_{
\tau},		
 \end{align*}
where $\tau=(\tau_0,\tau_1, \tau_2, \tau_3)$.

Due to 
$$\frac{\partial   }{\partial x_j}
        {  \frac{1}{\Gamma(\alpha_j)}  }  \int_{a_j}^{x_j} \frac{f  
(q_0, \dots, \tau_j, \dots, q_3)
}{(x_j-\tau_j)^{1-       {  \alpha_j   }  }} d\tau_j =
\frac{\partial   }{\partial x_j}    \sum_{n=0}^3  
       {  \frac{1}{\Gamma(\alpha_n)}  }  \int_{a_n}^{x_n} \frac{f  
(q_0, \dots, \tau_n, \dots, q_3)
}{(x_n-\tau_n)^{1-{\alpha_n}}} d\tau_n ,
 $$
one concludes that 
\begin{align*} 
& 
{}^{\psi}\mathfrak D_a^{\vec{\alpha}}[f](q,x) \\
= & 
\sum_{j=0}^3 \psi_j \frac{\partial   }{\partial x_j}
    \frac{1}{\Gamma({\alpha_j})} \int_{a_j}^{x_j} \frac{f(q_0, \dots, \tau_j, \dots, q_3)}{(x_j-\tau_j)^{1-{\alpha_j}}}d\tau_j \\
= & 
\sum_{j=0}^3 \psi_j \frac{\partial}{\partial x_j} \sum_{n=0}^3 \frac{1}{\Gamma({\alpha_n})} 
  \int_{a_n}^{x_n} \frac{f(q_0, \dots, \tau_n, \dots, q_3)}{(x_n-\tau_n)^{1-{\alpha_n}}}d\tau_n  \\
 =&  {}^{\psi}\mathcal D_x \ {}^{\psi}\mathcal I_a^x [f](q, x, {\vec{\alpha}}).
  \end{align*}
	
\item The proof of this identity is a consequence of the Fundamental Theorem in the Riemann-Liouville fractional derivative calculus:   
	\begin{align*} 
	{}^{\psi}\mathfrak D_a^{\vec{\alpha}}\circ{}^{\psi} \mathfrak I_a^{\vec{\alpha}}[f](q,x)  =&  \sum_{j=0}^3 \psi_j  D _{a_j^+}^{{\alpha_j}} \sum_{j=0}^3\frac{1}{\Gamma(\alpha_j)} \int_{a_j}^{x_j}\frac{f_j(q_0,\dots, \tau_j, \dots, q_3)}{ (x_j- \tau_j)^{1-\alpha_j}} d\tau_j \\  
	 = & \sum_{j=0}^3 \psi_j f_j (q_0, \dots, x_j , \dots  q_3) . 
 \end{align*} 

\item  
 \begin{align*} 
{}^{\bar \psi}\mathcal D_x \circ {}^{\psi}\mathfrak D_a^{\vec{\alpha}}[f](q,x) =&{}^{\bar \psi}\mathcal D_x \circ  {}^{\psi}\mathcal D_x \circ {}^{\psi}\mathcal I_a^x [f](q, x,\vec{\alpha}) = \Delta_{\mathbb R^4}\circ {}^{\psi} \mathcal I_a^x [f](q, x,\vec{\alpha}) .
\end{align*}

\item  \begin{align*} 
{}^{\psi}\mathfrak D_a^{\vec{\alpha}} \circ {}^{\psi}\mathfrak D_a^{\vec{\beta}}  [f](q,x) = & \sum_{j=0}^3 \psi_j D _{a_j^+}^{{\alpha_j}} \sum_{k=0}^3 \psi_k( D _{a_k^+}^{{\beta_j}}f)(q_0, \dots, x_k , \dots,  q_3) \\
= & \sum_{j,k=0}^3 \psi_j  \psi_k D _{a_j^+}^{{\alpha_j}} [D _{a_k^+}^{{\beta_j}}f(q_0, \dots, x_k , \dots, q_3)] \\
= & \sum_{j=0}^3 \psi_j^2  D _{a_j^+}^{{\alpha_j + \beta_j}} f(q_0, \dots, x_j , \dots, q_3)]       
,\\
{}^{\bar\psi}\mathfrak D_a^{\vec{\alpha}} \circ {}^{\psi}\mathfrak D_a^{\vec{\beta}} [f](q,x) = & \sum_{j=0}^3 \bar \psi_j D _{a_j^+}^{{\alpha_j}} \sum_{k=0}^3 \psi_k( D _{a_k^+}^{{\beta_j}} f)(q_0, \dots, x_k , \dots, q_3) \\
= & \sum_{j,k=0}^3 D _{a_j^+}^{{\alpha_j + \beta_j}}f(q_0, \dots, x_j , \dots,  q_3)].
\end{align*}
\end{enumerate}
\end{proof}
\begin{remark}\label{Remark1}
Properties exhibited by Proposition \ref{propFRACD} gives an extension of basic formulas related to the standard fractional Riemann-Louville derivative to the context of a fractional quaternionic analysis. This, essentially with $\dfrac{d}{dx}$ and  $({\bf I}_{a^+}^{\vec{\alpha}} f)(x)$ replaced by ${}^{\psi}\mathcal D$ and ${}^{\psi}\mathcal I_a^x [f](q, x,\vec{\alpha})$ respectively.

In particular, 2. establish a quaternionic analogous of the Fundamental Theorem, comparing with \eqref{FundTheorem}.

An additional observation is that operator ${}^{\psi}\mathfrak D_{r,a}^{\vec{\alpha}}$ (right action of ${}^{\psi}\mathfrak D_{a}^{\vec{\alpha}}$) meets similar properties given in Proposition \ref{propFRACD}. For example 
 \begin{align*} 
{}^{\psi}\mathfrak D_{r,a}^{\vec{\alpha}}[f](q,x)={}^{\psi} \mathcal D_{r,x} \ {}^{\psi}\mathcal I_a^x [f](q, x, {\vec{\alpha}}).
  \end{align*}
\end{remark}
	{
\begin{proposition}(Stokes type integral formula induced by  ${}^{\psi}\mathfrak D_{a}^{\vec{\alpha}}$)  
If $\vec{\alpha},\vec{\beta} \in\mathbb C^4$ with 
  $0< \Re\alpha_{\ell}, \Re\beta_{\ell}<1$ for $\ell=0,1,2,3$ and  let $f,g \in AC^1(\overline{J_a^b}, \mathbb H)$ consider $q\in J_a^b$ such that   
 the mappings $x\mapsto {}^{\psi}\mathcal I_a^x [f](q,x, \vec{\alpha})$ and $ x\mapsto {}^{\psi} \mathcal I_a^x [g](q,x, \vec{\beta} )$ belong to  $ C^1(\overline{J_a^b}, \mathbb H(\mathbb C))$.   
 Then 
 \begin{align*} &   \int_{\partial J_a^b} {}^{\psi} \mathcal I_a^x [g](q, x, \vec{\beta}) \sigma^{{\psi} }_x {}^{\psi}\mathcal I_a^x [f](q, x, \vec{\alpha}) \\ 
=  &      \int_{J_a^b }  \left( {}^{\psi} \mathcal I_a^x [g](q, x,\vec{\beta}) \ {}^{\psi}\mathfrak D_a^{\vec{\alpha}}[f](q, x) + \  {}^{\psi}\mathfrak D_{r,a}^{\vec{\beta}}[g](q, x) {}^{\psi}\mathcal I_a^x [f](q, x,\vec{\alpha})\right)dx.
\end{align*}
 \end{proposition}
\begin{proof}Applies \eqref{StokesHyp} to $ {}^{\psi}\mathcal I_a^x [g](q, x, \vec{\beta})$ and ${}^{\psi}\mathcal I_a^x [f](q, x,\vec{\alpha})$ to obtain 
\begin{align*} &   \int_{\partial \Omega} {}^{\psi}\mathcal I_a^x [g](q, x, \vec{\alpha}) \sigma^{{\psi}}_x {}^{\psi} \mathcal I_a^x [f](q, x, \vec{\alpha})  \\ 
= &  \int_{\Omega} \left( {}^{\psi} \mathcal I_a^x [g](q, x, \vec{\beta}) \ {}^{\psi}\mathfrak D_a^{\vec{\alpha}} [f](q, x) + \ 
  {}^{\psi}\mathfrak D_{r,a}^{\vec{\beta}}[g](q, x) {}^{\psi}\mathcal I_a^x [f](q, x,\vec{\alpha})\right)dx.
\end{align*} 
\end{proof}
}{ 
\begin{theorem}\label{B-P-F-D} (Borel-Pompieu type formula induced by ${}^{\psi}\mathfrak D_{a}^{\vec{\alpha}}$ and ${}^{\psi}\mathfrak D_{r,a}^{\vec{\beta}}$) 
Let $\vec{\alpha},\vec{\beta}\in\mathbb C^4$ with $0< \Re\alpha_{\ell}, \Re\beta_{\ell}<1$ for $\ell=0,1,2,3$ and $f,g \in AC^1(\overline{J_a^b}, \mathbb H)$. Consider $q\in J_a^b$ such that the mappings   
 $x\to {}^{\psi}\mathcal I_a^x [f](q,x,\vec{\alpha})$ and $x\to {}^{\psi} \mathcal I_a^x [g](q,x, \vec{\beta})$, for $x\in J_a^b$, belong to $C^1(\overline{J_a^b}, \mathbb H(\mathbb C))$ then 
\begin{align*}  &  
  \int_{\partial J_a^b } \left(\mathfrak K^{\vec{\alpha}}_{\psi, a}(\tau,x) \sigma_{\tau}^{\psi} \mathcal I_a^{\tau} [f](q,\tau, \vec{\alpha})  
		+    {}^{\psi}\mathcal I_a^{\tau} [g](q,\tau,\vec{\beta})  \sigma_{\tau}^{\psi} \mathfrak K^{\vec{\beta}}_{\psi, a}(\tau,x) \right) \\
		&		- 
\int_{J_a^b} \left(  \mathfrak K^{\vec{\alpha}}_{\psi, a}(y,x)  
  {}^{\psi}\mathfrak D_a^{\vec{\alpha}}[f](q,y)   +   
	   {}^{\psi}\mathfrak D_{r,a}^{\vec{\beta}}[g](q,y)	\mathfrak K^{\vec{\beta}}_{\psi, a}(y,x)  
    		\right)  dy      \\
		=  &    \left\{ \begin{array}{ll}  
		\displaystyle \sum_{i=0}^3 (f+g) (q_0, \dots, x_i, \dots, q_3) + N[f](q,x, \vec{\alpha}) + N[g](q,x, \vec{\beta}) , &    x\in 
		J_a^b ,  \\ 0 , &  x\in \mathbb H\setminus\overline{J_a^b},                    
	\end{array} \right. 
	\end{align*} 
where  
			$$ \mathfrak K^{\vec{\alpha}}_{\psi, a}(y,x):=
		 	\sum_{i=0}^3 \left[ D _{a_i^+}^{\alpha_i}    K_{\psi} (y-x) \right] $$ and  the partial derivative   
		$D _{a_i^+}^{\alpha_i} $
 is in terms of  real the  component $x_i$ of $x$ and  		  
  $$N[f ](q,x, \vec{\alpha}) = \displaystyle  \sum_{{{\begin{array}{r}i,j=0 \\ i\neq j \end{array}}}}^3  
\frac{ ({\bf I}_{a_j^+}^{\alpha_j} f  )
(q_0, \dots, x_j, \dots, q_3) }{ \Gamma[\alpha_i] (x-a_i)^{\alpha_i}}.$$
\end{theorem} 
\begin{proof}
Borel-Pompieu formula associated to the $\psi$-Fueter operator, see \eqref{BorelHyp}, applied in ${}^{\psi}\mathcal I_a^x [f](q,x,\vec{\alpha})$ and ${}^{\psi}\mathcal I_a^x [g](q,x, \vec{\beta})$
		gives  us 
 \begin{align*}  &  \int_{\partial J_a^b} (K_{\psi}(\tau-x)\sigma_{\tau}^{\psi}  {}^{\psi}\mathcal I_a^{\tau} [f](q,\tau, \vec{\alpha}) + {}^{\psi} \mathcal I_a^{\tau} [g](q,\tau,\vec{\beta}) \sigma_{\tau}^{\psi} K_{\psi}(\tau-x) ) \nonumber  \\ 
&  -   
\int_{J_a^b} (K_{\psi} (y-x) {}^{\psi}\mathcal D_{y} {}^{\psi} \mathcal I_a^{y} [f] (q,y, \vec{\alpha}) +	{}^{\psi}\mathcal D_{r,y} {}^{\psi}\mathcal I_a^{y} [g](q,y, \vec{\beta}) K_{\psi} (y-x))dy   \nonumber \\
		=  &      \left\{ \begin{array}{ll} {}^{\psi} \mathcal I_a^{x} [f](q,x,\vec{\alpha}) +  {}^{\psi}\mathcal I_a^{x} [g](q,x, \vec{\beta})  , &     x\in 
		J_a^b ,  \\ 0 , &  x\in \mathbb H\setminus\overline{J_a^b},                   
\end{array} \right. 
\end{align*}  
As
 \begin{align*} 
{}^{\psi}\mathfrak D_a^{\vec{\alpha}}[f](q,y) = &  {}^{\psi}\mathcal D_y  {}^{\psi} \mathcal I_a^y [f](q, y, \vec{\alpha}),\\ 
{}^{\psi}\mathfrak D_{r,a}^{\vec{\beta}}[g](q,y)  
 = &   {}^{\psi}\mathcal D_{r,y} {}^{\psi}\mathcal I_a^y [g](q, y,\vec{\beta}), 
  \end{align*}
then
 \begin{align*}  &  \int_{\partial J_a^b} (K_{\psi}(\tau-x)\sigma_{\tau}^{\psi} {}^{\psi}\mathcal I_a^{\tau} [f](q,\tau, \vec{\alpha}) + {}^{\psi}\mathcal I_a^{\tau} [g](q,\tau,\vec{\beta})    \sigma_{\tau}^{\psi} K_{\psi}(\tau-x) ) \nonumber  \\ 
&  - 
\int_{J_a^b} (K_{\psi} (y-x) {}^{\psi}\mathfrak D_a^{\vec{\alpha}}[f](q,y) + {}^{\psi}\mathfrak D_{r,a}^{\vec{\beta}}[g](q,y) K_{\psi} (y-x)
     )dy   \nonumber \\
		=  &      \left\{ \begin{array}{ll} {}^{\psi} \mathcal I_a^{x} [f](q,x,\vec{\alpha}) +{}^{\psi} \mathcal I_a^{x} [g](q,x, \vec{\beta})  , &  x\in
		J_a^b ,   \\ 0 , &  x\in \mathbb H\setminus\overline{J_a^b}, 
\end{array} \right. 
\end{align*} 
Particularly, for $g=0$ acting $\sum_{i=0}^3 D _{a_i^+}^{\alpha_i}$, where $D _{a_i^+}^{\alpha_i} $ is given in terms of the real component $x_i$ of $x$, on both sides, we see that
 \begin{align*}  & \sum_{i=0}^3 D _{a_i^+}^{\alpha_i}  \left[ 
  \int_{\partial J_a^b }  ( K_{\psi}(\tau-x)\sigma_{\tau}^{\psi} {}^{\psi}\mathcal I_a^{\tau} [f](q,\tau,\alpha)  
		\right]    \\ -  &
			\displaystyle \sum_{i=0}^3 D _{a_i^+}^{\alpha_i}  \left[
\int_{J_a^b}  K_{\psi} (y-x)
  {}^{\psi}\mathfrak D_a^{\alpha}[f](q,y)     dy   \right]    \\
		=  &      \left\{ \begin{array}{ll} 	\displaystyle \sum_{i=0}^3 D _{a_i^+}^{\alpha_i}  {}^{\psi} \mathcal I_a^{x} [f](q,x,\alpha)   , &    x\in
		J_a^b ,  \\ 0 , &  x\in \mathbb H\setminus\overline{J_a^b} ,                    \end{array} \right. 
		\end{align*} 
Combining fundamental theorem for Riemann-Liouville fractional calculus and (\ref{cte}) we obtain the following:  
		\begin{align*} \sum_{i=0}^3  D _{a_i^+}^{\alpha_i}    {}^{\psi}\mathcal I_a^{x} [f] (q,x,\vec{\alpha}) = 
		& \sum_{i=0}^3 \sum_{j=0}^3   D _{a_i^+}^{\alpha_i} [ ({\bf I}_{a_j^+}^{\alpha_j}f)(q_0, \dots, x_j, \dots, q_3) ] \\
		= &  \sum_{i=0}^3     D _{a_i^+}^{\alpha_i} [ ({\bf I}_{a_i^+}^{\alpha_i} f )(q_0, \dots, x_i, \dots, q_3) ]  \\  
		  & +   \sum_{{\  { \begin{array}{r}i,j=0 \\ i\neq j \end{array}} }}^3 D _{a_i^+}^{\alpha_i} [ ({\bf I}_{a_j^+}^{\vec{\alpha}} f )(q_0, \dots, x_j, \dots, q_3) ] \\
= &  \sum_{i=0}^3   f (q_0, \dots, x_i, \dots, q_3) \\
  &     +   \sum_{{\  { \begin{array}{r}i,j=0 \\ i\neq j \end{array}} }}^3  
\frac{({\bf I}_{a_j^+}^{\alpha_j} f )
(q_0, \dots, x_j, \dots, q_3) }{ \Gamma[\alpha_i] (x-a_i)^{\alpha_i}}    ,
		\end{align*}
for all $x\in J_a^b$.
				
Finally, Fubbini' Theorem,  Leibniz formula and the previous computations give us that
			 \begin{align*}  &  
  \int_{\partial J_a^b}  \left[ \sum_{i=0}^3 D _{a_i^+}^{\alpha_i}   K_{\psi}(\tau-x)  \right]  \sigma_{\tau}^{\psi} {}^{\psi}\mathcal I_a^{\tau} [f](q,\tau, \vec{\alpha})  \\
		  - &
\int_{J_a^b}  \left[ \sum_{i=0}^3  D _{a_i^+}^{\alpha_i}  K_{\psi} (y-x) \right] 
  {}^{\psi}\mathfrak D_a^{\vec{\alpha}}[f](q,y)dy  \\
		=  &      \left\{ \begin{array}{ll} 
	\begin{array}{l} 
\displaystyle \sum_{i=0}^3 f (q_0, \dots, x_i, \dots, q_3)   \\
      +  \displaystyle  \sum_{{{\begin{array}{r}i,j=0 \\ i\neq j \end{array}}}}^3  
\frac{({\bf I}_{a_j^+}^{\alpha_j} f )(q_0, \dots, x_j, \dots, q_3) }{ \Gamma[\alpha_i] (x-a_i)^{\alpha_i}}    ,  
	\end{array}
		, &    x\in	J_a^b ,  \\ 0 , &  x\in \mathbb H\setminus\overline{J_a^b}.                    
\end{array} \right. 
\end{align*} 
For $f=0$, we can repeat the argument to conclude that 
				\begin{align*}  &  
  \int_{\partial J_a^b} {}^{\psi} \mathcal I_a^{\tau} [g](q,\tau, \vec{\beta})  \sigma_{\tau}^{\psi} \mathfrak K^{\vec{\beta}}_{\psi, a}(\tau,x)     		  - 
\int_{J_a^b} {}^{\psi}\mathfrak D_{r,a}^{\vec{\beta}}[g](q,y) \mathfrak K^{\vec{\beta}}_{\psi, a}(y,x)dy \\
		=  &      \left\{ \begin{array}{ll}  
\displaystyle 		\sum_{i=0}^3  g (q_0, \dots, x_i, \dots, q_3) +  N[g](q,x, \vec{\beta})
      		, &    x\in 
		J_a^b ,  \\ 0 , &  x\in \mathbb H\setminus\overline{J_a^b}.                 
\end{array} \right. 
\end{align*} 
\end{proof}
\begin{remark} According to decomposition of the hiperholomorphic  Cauchy kernel in terms of Gegenbauer polynomials given in page 93 of \cite{GS2} we see that  
$$ \mathfrak K^{\vec{\alpha}}_{\psi, a}(y,x):=
		\frac{1}{2\pi^2} \sum_{k=0}^{\infty}  \frac{1}{|y|^{k+3}}	\sum_{i=0}^3  D _{a_i^+}^{\alpha_i} 
		\left[  |x|^k A_{4,k} (x,y) \right] ,$$
		with
		$$2A_{4,k} (x,y):= [ (k+1)  C_{k+1}^{1} (s) +  (2-n) C_{k}^{2} (s) \omega_y \wedge \omega_x]  \bar \omega_x, $$
where  $C_{k+1}^{1} $ and   $ C_{k}^{2} $  are the  Gegenbauer polynomials, $x= |x|\omega_x$, $y= |y|\omega_y$ and   $s=(\omega_x, \omega_y)$, see  \cite{GS2}. 
\end{remark}
}
    {
\begin{corollary}  Under the same the hypothesis of Proposition \ref{B-P-F-D} we have: 
\begin{align*}  &  
  \int_{\partial J_a^b} \left(\mathfrak K^{\vec{\alpha}}_{\psi, a}(\tau,q)\sigma_{\tau}^{\psi} {}^{\psi} \mathcal I_a^{\tau} [f](q,\tau,\vec{\alpha})  
		+   {}^{\psi} \mathcal I_a^{\tau} [g](q,\tau, \vec{\beta})  
		\sigma_{\tau}^{\psi} \mathfrak K^{\vec{\beta}}_{\psi, a}(\tau,q)\right) \\
		&		- 
\int_{J_a^b} \left(\mathfrak K^{\vec{\alpha}}_{\psi, a}(y,q) {}^{\psi}\mathfrak D_a^{\vec{\alpha}}[f](q,y)   +   
	   {}^{\psi}\mathfrak D_{r,a}^{\vec{\beta}}[g](q,y)	\mathfrak K^{\vec{\beta}}_{\psi, a}(y,q)\right)dy \\
		=  &  4(f+g) (q) +  N[f](q,q, \vec{\alpha})  + N[ g](q,q, \vec{\beta}).
\end{align*}
If ${}^{\psi}\mathfrak D_a^{\vec{\alpha}}[f](q,\cdot) = {}^{\psi}\mathfrak D_{r,a}^{\vec{\alpha}}[g](q,\cdot) =0$ on $J_a^b$, then   
\begin{align*}  &  
  \int_{\partial J_a^b} \left(\mathfrak K^{\vec{\alpha}}_{\psi, a}(\tau,x)  \sigma_{\tau}^{\psi} {}^{\psi}\mathcal I_a^{\tau} [f](q,\tau,\vec{\alpha})  
		+   {}^{\psi} \mathcal I_a^{\tau} [g](q,\tau,\beta)  \sigma_{\tau}^{\psi} \mathfrak K^{\vec{\beta}}_{\psi, a}(\tau,x)  \right) \\
		=  &      \left\{ \begin{array}{ll}  
		\displaystyle \sum_{i=0}^3 (f+g) (q_0, \dots, x_i, \dots, q_3) +  N[f](q,x,\vec{\alpha}) + N[g](q,x,\vec{\beta})
      		, &    x\in 
		J_a^b ,  \\ 0 , &  x\in \mathbb H\setminus\overline{J_a^b}                      
\end{array} \right. 
\end{align*}
and for $x=q$ we have that
\begin{align*}  &  
  \int_{\partial J_a^b} \left(\mathfrak K^{\vec{\alpha}}_{\psi, a}(\tau,q)  \sigma_{\tau}^{\psi} {}^{\psi} \mathcal I_a^{\tau} [f](q,\tau, \vec{\alpha})  
		+   {}^{\psi} \mathcal I_a^{\tau} [g](q,\tau,\vec{\beta}) \sigma_{\tau}^{\psi} \mathfrak K^{\vec{\beta}}_{\psi, a}(\tau,q)\right) \\
		=  &  
4(f+g) (q) + N[f](q,q, \vec{\alpha}) + N[g](q,q,\vec{\beta}). 
\end{align*}
\end{corollary} 
}
\subsection{An iterated fractional $\Psi$-Fueter operator}\label{3.2}
Now we shall study properties of an iterated fractional $\Psi$-Fueter operator. 

Let $n\in \mathbb N$ and let  $\Psi:= \{\psi_1, \dots, \psi_n\}$ be the collection of $n$ structural sets of $\mathbb H$.  The left and right
 iterated $\Psi$-Fueter operators are denoted by 
$$ {}^{\Psi}\mathcal D   = {}^{\psi_1}\mathcal D  \circ \cdots \circ {}^{\psi_n}\mathcal D,$$
and
$$ {}^{\Psi}\mathcal D _r = {}^{\psi_1}\mathcal D_r  \circ \cdots \circ {}^{\psi_n}\mathcal D_r$$
respectively. We will restrict our attention to the case $\Psi:= \{\psi, \dots, \psi\}$. Motivated by \cite{BSGS}, it is possible to consider the general situation, but we will not develop this point here.

\begin{definition} Let $\vec {\alpha} = (\alpha_0, \alpha_1,\alpha_2,\alpha_3) \in\mathbb C^4$ with  $n= [\Re(\alpha_{\ell})]+1 $ for $\ell=0,1,2,3$. We define the iterated fractional $\Psi$-Fueter operator of order $\vec{\alpha}$ on $AC^n(\overline{J_a^b}, \mathbb H)$ by
\begin{align*} 
{}^{\Psi}\mathfrak D_a^{\vec{\alpha}}[f] (q,x)= {{}^{\Psi}\mathcal D_x } \circ {}^{\psi}  \mathcal I_a^x [f] (q,x, n-\vec{\alpha}), 
 \end{align*}
for all $q,x\in J_a^b$    and similarly 
\begin{align*} 
{}^{\Psi}\mathfrak D_{r,a}^{\vec{\alpha}}[f](q,x) = {}^{\Psi}\mathcal D_{r,x}  \circ {}^{\psi}\mathcal I_a^x [f](q, x,n-\vec{\alpha}), 
\end{align*}
\end{definition}
Here and subsequently   by $\vec 1$ we mean the vector $(1,1,1,1)$. Hence 
$n\vec 1-\vec{\alpha}=(n-\alpha_0, n-\alpha_1,n-\alpha_2,n-\alpha_3)$.

\begin{proposition} 
Consider  $ \vec{\alpha} = (\alpha_0, \alpha_1,\alpha_2,\alpha_3) \in\mathbb C^4$  with  $n= [\Re(\alpha_{\ell})] +1  $ for $\ell=0,1,2,3$. Then 
\begin{enumerate} 	
\item If 
 $f \in  C (J_a^b,\mathbb H)\cup L_2(J_a^b,\mathbb H)$ then
\begin{align*}  &  \displaystyle 
{}^{\Psi}\mathfrak D_a^{\vec{\alpha}} \circ {}^{\psi}\mathfrak I_a^{n\vec 1-\vec{\alpha}} \circ {}^{\psi} \mathcal T ^{(n-1)}[f](q,x) \\ 
= &   \frac{1}{2}  \sum_{j=0}^3  f (q_0,\dots, x_j, \dots, q_3)  +
   \frac{1}{2} \sum_{j=0}^3  {}^{\psi} \mathcal D ^{(n-1)} \circ \overline{ {}^{\psi} \mathcal T ^{(n-1)} [f] }(q_0,\dots, x_j, \dots, q_3) \psi_j  , \end{align*}
where ${}^{\psi} \mathcal T ^{(n-1)}  = {}^{\psi} \mathcal T  \circ \cdots \circ  {}^{\psi} \mathcal T $,   $(n-1)$-times. 	

\item If 
 $f \in AC^n(J_a^b,\mathbb H)$ then  
$$\displaystyle {}^{\bar \psi}\mathcal D_x^{(n)} \circ {}^{\Psi}\mathfrak D_a^{\vec{\alpha}}[f](q,x) = \Delta_{\mathbb R^4}^n\circ {}^{\psi} \mathcal I_a^x [f](q, x,\vec{\alpha}),$$
where $ {}^{\bar \psi}\mathcal D_x^{(n)} =  {}^{\bar \psi}\mathcal D_x \circ \cdots \circ {}^{\bar \psi}\mathcal D_x$ and  
$ \Delta_{\mathbb R^4}^n =  \Delta_{\mathbb R^4}\circ \cdots \circ  \Delta_{\mathbb R^4} $  $n$-times. 

\item Let $\vec{\beta}=(\beta_0, \dots, \beta_3)\in \mathbb C^4$ with $m= [\Re(\beta_{\ell})] +1$ for $\ell=0,1,2,3$.  Let $f \in AC^{n+m}(J_a^b,\mathbb H)$ be such that  the mapping $x\to {}^{\psi}\mathcal I_a^{x} [f](q, x,\vec{\alpha})$ belongs to $C^{n+m}(J_a^b, \mathbb H)$  then    

 \begin{align*} 
{}^{\Psi}\mathfrak D_a^{\vec{\alpha}} \circ {}^{\Psi}\mathfrak D_a^{\vec{\beta}}[f]  (q,x) = & \sum_{j=0}^3 \psi_j^{m+n} D _{a_j^+}^{{\alpha_j +  \beta_j}} f(q_0, \dots, x_j , \dots,  q_3)]   
\end{align*} 
and 
 \begin{align*} 
{}^{\bar\psi}\mathfrak D_a^{\vec{\alpha}} \circ {}^{\Psi}\mathfrak D_a^{\vec{\beta}}[f] (q,x) = & \sum_{j,k=0}^3 
 \psi_j^{-n+m} D _{a_j^+}^{{\alpha_j + \beta_j}} f(q_0, \dots, x_j , \dots,  q_3).
  \end{align*}
\end{enumerate}
\end{proposition}
\begin{proof} 
\begin{enumerate} 
\item Note that 
\begin{align*} 
&  {}^{\Psi}\mathfrak D_a^{\vec{\alpha}} \circ  {}^{\psi} \mathfrak I_a^{n\vec 1-\vec{\alpha}} \circ {}^{\psi} \mathcal T ^{(n-1)}[f](q,x) \\
 = & ( {}^{\psi} \mathcal D ^{(n-1)} \circ {}^{\psi} \mathcal D \circ {}^{\psi} \mathcal I_a^x)[{}^{\psi} \mathfrak I_a^{n\vec 1-\vec{\alpha}} [ {}^{\psi} \mathcal T ^{(n-1)}(f)] ] (q, x, n\vec 1-\vec{\alpha}) \\
  = & {}^{\psi} \mathcal D ^{(n-1)} \circ ( {}^{\Psi}\mathfrak D_a^{n\vec 1-\vec{\alpha}} \circ {}^{\psi} \mathfrak I_a^{n\vec 1-\vec{\alpha}}) [{}^{\psi} \mathcal T ^{(n-1)}(f)](q, x ) \\
  =  &  {}^{\psi} \mathcal D ^{(n-1)} [\sum_{j=0}^3 \psi_j ( {}^{\psi} \mathcal T ^{(n-1)}(f) \ )_j(q_0,\dots, x_j, \dots, q_3)] \\
 = & \frac{1}{2} {}^{\psi} \mathcal D ^{(n-1)}[\sum_{j=0}^3 \psi_j ( \bar \psi_j {}^{\psi} \mathcal T ^{(n-1)}(f) (q_0,\dots, x_j, \dots, q_3) \\
 &  +  \sum_{j=0}^3 \overline{ {}^{\psi} \mathcal T ^{(n-1)}(f)} (q_0,\dots, x_j, \dots, q_3)\psi_j]  \\
  =  &  \frac{1}{2}  [\sum_{j=0}^3   {}^{\psi} \mathcal D ^{(n-1)} \circ {}^{\psi} \mathcal T ^{(n-1)}(f) (q_0,\dots, x_j, \dots, q_3) \\ 
 &   +
  \sum_{j=0}^3 {}^{\psi} \mathcal D ^{(n-1)} \circ  \overline{ {}^{\psi} \mathcal T ^{(n-1)}(f)} (q_0,\dots, x_j, \dots, q_3) \psi_j ]  \\
  =  & \frac{1}{2}[\sum_{j=0}^3  f (q_0,\dots, x_j, \dots, q_3) \\ 
   &  +
  \sum_{j=0}^3 {}^{\psi} \mathcal D ^{(n-1)} \circ\overline{{}^{\psi} \mathcal T ^{(n-1)}(f) }(q_0,\dots, x_j, \dots, q_3)\psi_j]. 
\end{align*} 
The proof is completed by using Proposition \ref{propFRACD} and \eqref{FueterInv}.

\item The verification of 2. and 3. are reduced to direct computations. 
\end{enumerate} 
\end{proof}

\begin{proposition}\label{B-P-H-O} (The Borel-Pompieu formula of higher order)
Let $J_1 , \dots , J_n\subset\mathbb H$ be a sequence of open bounded rectangles such that $ J_{k }\supset \overline{J_{k+1}}$  for $k=1,\dots, n-1$.  
Let $f:\overline{J_1}\to \mathbb H$ such that ${}^{\psi}\mathcal D^{(n-{\ell} )} [f]\in C^1({J_{\ell}}, \mathbb H)\cap C (\overline{J_{\ell}}, \mathbb H) $  for ${\ell}=1,\dots, n$ then 
\begin{align*}  &  \int_{\partial J_n} K_{\psi}(y_n  -x)\sigma_{y_n}^{\psi}f(y_n)   \\
   - &
\displaystyle \int_{J_n \times \partial J_{n-1}} K_{\psi} (y_n -x)
     K_{\psi}(y_{n-1}-y_n)\sigma_{y_{n-1}}^{\psi}  \ {}^{\psi}\mathcal D^{(1)}[f](y_{n-1})d\mu_{y_n} 
  \\
	+ & \displaystyle  \int_{J_n  \times J_{n-1} \times \partial J_{n-2}} K_{\psi} (y_n -x) K_{\psi} (y_{n-1} -y_n)
  		    K_{\psi}(y_{n-2} -y_{n-1})\sigma_{y_{n-2}}^{\psi} \\
					&  \  \hspace{3cm}  {}^{\psi}\mathcal D^{(2 )}[f](y_{n-2})d\mu_{y_{n-1}}d\mu_{y_n}  \\ 
						 +			& \\
					& \vdots					
					  \\
	+ & (-1)^{n-1}  \displaystyle  \int_{J_n  \times J_{n-1}\times \cdots  \times J_2 \times \partial J_1 } K_{\psi} (y_n -x) 
	K_{\psi} (y_{n-1}  -y_n) \cdots  K_{\psi} (y_2  -y_3)  \\
		 & \  \hspace{2cm} 
  		    K_{\psi}(y_1  -y_2)\sigma_{y_1}^{\psi}  \ {}^{\psi}\mathcal D^{(n -1)}[f](y_1)d\mu_{y_2}d\mu_{y_3}\cdots d\mu_{y_{n-1}}d\mu_{y_n}
		\\	
 +  & (-1)^{n} \displaystyle  \int_{J_{n}\times J_{n-1} \times \cdots \times J_2\times J_1} 
 K_{\psi} (y_n  -x)  K_{\psi} (y_{n-1}  -y_n) \cdots K_{\psi} (y_2  -y_3) \\
 & \  \hspace{2cm} 
  K_{\psi} (y_1  -y_2) {}^{\psi}\mathcal D^{(n  )} [f] (y_1  ) d\mu_{y_1} 
			d\mu_{y_2} \cdots d\mu_{y_{n-1}}d\mu_{y_n } \\
		=  &      \left\{ \begin{array}{ll}  f(x), &  x\in J_n  ,  \\ 0 , &  x\in \mathbb H\setminus\overline{J_n},                    
\end{array} \right. 
\end{align*} 
\end{proposition}
 \begin{proof}
Fixing each $\ell=1,\dots,n$, we can assert that 
\begin{align*}  &  \int_{\partial J_{\ell}} K_{\psi}(y_{\ell}-x)\sigma_{y_{\ell}}^{\psi}  \ {}^{\psi}\mathcal D^{(n -{\ell}) }[f](y_{\ell})  
   - 
\int_{J_{\ell} }  K_{\psi} (y_{\ell}-x) {}^{\psi}\mathcal D^{(n+1-{\ell})} [f] (y_{\ell})d\mu_{y_{\ell}}  \\
		=  &      \left\{ \begin{array}{ll}  {}^{\psi}\mathcal D^{(n -{\ell})} f(x), &  x\in J_{\ell},  \\ 0 , &  x\in \mathbb H\setminus\overline{J_{\ell}},                    
\end{array} \right. 
\end{align*} 
where ${}^{\psi}\mathcal D^{0} = {}^{\psi}\mathcal D_r^{0}$ is  the identity operator.
		
Particularly, for ${\ell}=1$ we obtain 
\begin{align*}  &  \int_{\partial J_1}   K_{\psi}(y_1 -x)\sigma_{y_1}^{\psi} \ {}^{\psi}\mathcal D^{(n -1)}[f](y_1)  
   - 
\int_{J_1}  K_{\psi} (y_1  -x) {}^{\psi}\mathcal D^{(n)} [f] (y_1)d\mu_{y_1}  \\
		=  &      \left\{ \begin{array}{ll}  {}^{\psi}\mathcal D^{(n -1)} f(x), &  x\in J_1  ,  \\ 0 , &  x\in \mathbb H\setminus\overline{J_1} ,                    
\end{array} \right. 
\end{align*} 
for  ${\ell}=2$:
\begin{align*}  &  \int_{\partial J_2} K_{\psi}(y_2 -x)\sigma_{y_2}^{\psi}  \ {}^{\psi}\mathcal D^{(n -2) }[f](y_2)  
   - 
\int_{J_2   }  K_{\psi} (y_2 -x)
  {}^{\psi}\mathcal D^{(n-1)}[f] (y_2)d\mu_{y_2}  \\
		=  &      \left\{ \begin{array}{ll}  {}^{\psi}\mathcal D^{(n -2)}f(x), &  x\in J_2  ,  \\ 0 , &  x\in \mathbb H\setminus\overline{J_2  } ,                    
\end{array} \right. 
\end{align*} 
and for  ${\ell}=3$
\begin{align*}  &  \int_{\partial J_3} K_{\psi}(y_3 -x)\sigma_{y_3}^{\psi}  \ {}^{\psi}\mathcal D^{(n -3) }[f](y_3)  
   - 
\int_{J_3}  K_{\psi} (y_3 -x)
  {}^{\psi}\mathcal D^{(n-2)} [f] (y_3)d\mu_{y_3}  \\
		=  &      \left\{ \begin{array}{ll}  {}^{\psi}\mathcal D^{(n -3)}f(x), &  x\in J_3  ,  \\ 0 , &  x\in \mathbb H\setminus\overline{J_3  } .                    
\end{array} \right. 
\end{align*} 
Combining the previous representation formulas yields
\begin{align*}  &  \int_{\partial J_3} K_{\psi}(y_3 -x)\sigma_{y_3}^{\psi}  \ {}^{\psi}\mathcal D^{(n -3) }[f](y_3)   \\
   - &
\int_{J_3 \times \partial J_2}  K_{\psi} (y_3 -x)
     K_{\psi}(y_2  -y_3)\sigma_{y_2}^{\psi}  \ {}^{\psi}\mathcal D^{(n -2) }[f](y_2)d\mu_{y_3  } 
  \\
	+ & \int_{J_3  \times J_2 \times \partial J_1}  K_{\psi} (y_3  -x)  K_{\psi} (y_2  -y_3)
  		    K_{\psi}(y_1  -y_2)\sigma_{y_1}^{\psi}  \ {}^{\psi}\mathcal D^{(n -1)   }[f](y_1  )  d\mu_{y_2  } 	d\mu_{y_3  }  \\ 
 -  & \int_{J_3\times J_2\times J_1}  K_{\psi} (y_3  -x)  K_{\psi} (y_2  -y_3) K_{\psi} (y_1  -y_2) {}^{\psi}\mathcal D^{(n)}[f](y_1)d\mu_{y_1} 
			d\mu_{y_2}d\mu_{y_3}  \\
		=  &     \left\{ \begin{array}{ll}  {}^{\psi}\mathcal D^{(n -3)}f(x), &  x\in J_3  ,  \\ 0 , &  x\in \mathbb H\setminus\overline{J_3}.  
\end{array} \right. 
\end{align*} 
We next proceed by induction to obtain the result.
\end{proof}

\begin{remark}In the same manner we can see that, given $g:\overline{J_1}\to \mathbb H$ such that
$${}^{\psi}\mathcal D_r^{(n -{\ell})}[g]\in C^1({J_{\ell}}, \mathbb H)\cap C (\overline{J_{\ell}}, \mathbb H),$$  
for ${\ell}=0,\dots, n$, we get
\begin{align*}  &  \int_{\partial J_n}     
g(y_n)
\sigma_{y_n}^{\psi} K_{\psi}(y_n -x) \\
   - &
\displaystyle \int_{J_n \times \partial J_{n-1}} {}^{\psi}\mathcal D_r^{(1)}[g](y_{n-1})  \sigma_{y_{n-1}}^{\psi} K_{\psi}(y_{n-1}-y_n) 
 K_{\psi} (y_n  -x) d\mu_{y_n} 
  \\
	+ & \displaystyle  \int_{J_n  \times J_{n-1} \times \partial J_{n-2}}	{}^{\psi}\mathcal D_r^{(2 )}[g](y_{n-2})
	\sigma_{y_{n-2}  }^{\psi} K_{\psi}(y_{n-2}  -y_{n-1})	K_{\psi} (y_{n-1} -y_n)
									\\
					&  \  \hspace{3cm} 
						K_{\psi} (y_n -x) d\mu_{y_{n-1}}	d\mu_{y_n}  \\ 
						 +			& \\
					& \vdots					
					  \\
	+ & (-1)^{n-1}  \displaystyle  \int_{J_n  \times J_{n-1}\times \cdots  \times J_2 \times \partial J_1}  
	{}^{\psi}\mathcal D_r^{(n -1)}[g](y_1) K_{\psi}(y_1 -y_2)\sigma_{y_1}^{\psi} 
	\\
	& \  \hspace{2cm} 
	 K_{\psi} (y_2  -y_3) \cdots K_{\psi} (y_{n-1}  -y_n) 
		K_{\psi} (y_n  -x) 
					d\mu_{y_2  } 	d\mu_{y_3  }\cdots   d\mu_{y_{n-1}  } 	d\mu_{y_n  }
		\\	
 +  & (-1)^{n} \displaystyle  \int_{J_{n}\times J_{n-1} \times \cdots \times J_2\times J_1   }
{}^{\psi}\mathcal D_r^{(n  )} [g] (y_1  ) 
  K_{\psi} (y_1  -y_2)\\
	 &  \  \hspace{2cm} 
K_{\psi} (y_2  -y_3)
\cdots 
  K_{\psi} (y_{n-1} -y_n)
 K_{\psi} (y_n  -x)  
     d\mu_{y_1} 
			d\mu_{y_2}  \cdots 	d\mu_{y_{n-1}  } 
		d\mu_{y_n } \\
		=  &      \left\{ \begin{array}{ll}    g(x)   , &  x\in J_n  ,  \\ 0 , &  x\in \mathbb H\setminus\overline{J_n  } ,                    \end{array} \right. 
		\end{align*} 
\end{remark}

\begin{corollary}
Let $J_1 , \dots , J_n\subset\mathbb H$ be a sequence of open bounded rectangles such that 
$ J_{k}\supset \overline{J_{k+1}}$  for $k=0,\dots n$ and set $f:\overline{J_1}\to \mathbb H$ 
such that  such    $g_q(x)= {}^{\psi}\mathcal I_a^{x} [f](q, x, \alpha )  $, for all $x\in J_1$, satisfies  
the hypothesis of Proposition  \ref{B-P-F-D}, i.e., ${}^{\psi}\mathcal D^{(n   -{\ell} )} [g_q]\in C^1({J_{\ell}}, \mathbb H)\cap C (\overline{J_{\ell}}, \mathbb H) $  for ${\ell}=1,\dots, n$ and 
 also  $ f\mid_{\bar{ J}_n} \in AC^1(\overline{J_n}, \mathbb H)$ satisfies that the mapping $x\to {}^{\psi} \mathcal I_a^x [f](q,x,\vec{\alpha})$, 
 for all $x\in J_a^b = J_n$, belongs to  $C^1(\overline{J_n}, \mathbb H(\mathbb C))$.  Then
\begin{align*}  &  \int_{\partial J_n} \mathfrak K^{\vec{\alpha}}_{\psi, a} (y_n -x)\sigma_{y_n}^{\psi} {}^{\psi} \mathcal I_a^{y_{n}} [f](q,y_{n}, \vec{\alpha})   \\
   - &
\displaystyle \int_{J_n \times \partial J_{n-1}} \mathfrak K^{\vec{\alpha}}_{\psi, a} (y_n -x)
     K_{\psi}(y_{n-1} -y_n)\sigma_{y_{n-1}}^{\psi}  \
		{}^{\Psi}\mathfrak D_a^{\vec{\alpha}-(n-1)\vec 1}[f](q,y_{n-1})d\mu_{y_n } 
  \\
	+ & \displaystyle  \int_{J_n  \times J_{n-1} \times \partial J_{n-2}} \mathfrak K^{\vec{\alpha}}_{\psi, a} (y_n -x) K_{\psi} (y_{n-1} -y_n)
  		    K_{\psi}(y_{n-2} -y_{n-1})\sigma_{y_{n-2}}^{\psi} \\
					&  \  \hspace{3cm} 		{}^{\Psi}\mathfrak D_a^{\vec{\alpha}-(n-2)\vec 1}[f](q,y_{n-2})d\mu_{y_{n-1}}d\mu_{y_n}  \\ 
		 +			& \\
					 & \vdots					
					  \\
	+ & (-1)^{n-1}  \displaystyle  \int_{J_n  \times J_{n-1}\times \cdots  \times J_2 \times \partial J_1}  
\mathfrak K^{\alpha}_{\psi, a} (y_n -x) K_{\psi} (y_{n-1}  -y_n) \cdots  K_{\psi} (y_2  -y_3)  \\
		 & \  \hspace{2cm} 
  		    K_{\psi}(y_1 -y_2)\sigma_{y_1}^{\psi}  \ {}^{\Psi}\mathfrak D_a^{\vec{\alpha}- \vec 1}[f](q,y_1) d\mu_{y_2}d\mu_{y_3}\cdots d\mu_{y_{n-1}}d\mu_{y_n}
		\\	
 +  & (-1)^{n} \displaystyle  \int_{J_{n}\times J_{n-1} \times \cdots \times J_2\times J_1} 
 \mathfrak K^{\alpha}_{\psi, a} (y_n  -x)  K_{\psi} (y_{n-1} -y_n) \cdots K_{\psi} (y_2  -y_3) \\
 & \  \hspace{2cm} 
  K_{\psi} (y_1  -y_2) \  {}^{\Psi}\mathfrak D_a^{\vec{\alpha}}[f](q,y_1) d\mu_{y_1} 
			d\mu_{y_2  }  \cdots 	d\mu_{y_{n-1}} d\mu_{y_n} \\
		=  &      \left\{ \begin{array}{ll}  
		\displaystyle \sum_{i=0}^3 f(q_0, \dots, x_i, \dots, q_3) + N[f](q,x, \vec{\alpha}) 
      		, &    x\in 
		J_n ,  \\ 0 , &  x\in \mathbb H\setminus\overline{J_n}.                    
\end{array} \right. 
\end{align*} 

\end{corollary}
\begin{proof}Applying Proposition \ref{B-P-H-O} to the mapping $x\to  {}^{\psi}\mathcal I_a^{x} [f](q, x, \vec{\alpha})$ and acting on both sides the operator 
$\sum_{i=0}^3  D _{a_i^+}^{\alpha_i} $ the proof is complete.
\end{proof}

\begin{corollary}The following assertions may be proved in much the same way as before.
\begin{enumerate}
\item If  $q\in J_n$ doing  $x=q$ we obtain  
\begin{align*}  &  \int_{\partial J_n}   \mathfrak K^{\vec{\alpha}}_{\psi, a} (y_n -q)\sigma_{y_n}^{\psi} {}^{\psi} \mathcal I_a^{y_{n}} [f](q,y_{n}, \vec{\alpha})  \\
   - &
\displaystyle \int_{J_n \times \partial J_{n-1}} \mathfrak K^{\vec{\alpha}}_{\psi, a} (y_n -q) K_{\psi}(y_{n-1} -y_n)\sigma_{y_{n-1}}^{\psi}  \
		{}^{\Psi}\mathfrak D_a^{\vec{\alpha}-(n-1)\vec 1}[f](q,y_{n-1}) d\mu_{y_n} 
  \\
	+ & \displaystyle  \int_{J_n  \times J_{n-1} \times \partial J_{n-2}} \mathfrak K^{\vec{\alpha}}_{\psi, a}(y_n -q) K_{\psi} (y_{n-1} -y_n)
  		    K_{\psi}(y_{n-2} -y_{n-1})\sigma_{y_{n-2}}^{\psi} \\
					&  \  \hspace{3cm} 	{}^{\Psi}\mathfrak D_a^{\vec{\alpha}-(n-2)\vec 1}[f](q,y_{n-2}) d\mu_{y_{n-1}}d\mu_{y_n}  \\ 
		 +			& \\
					 & \vdots					
					  \\
	+ & (-1)^{n-1} \displaystyle  \int_{J_n  \times J_{n-1}\times \cdots  \times J_2 \times \partial J_1}  
\mathfrak K^{\vec{\alpha}}_{\psi, a} (y_n -q) K_{\psi} (y_{n-1} -y_n) \cdots  K_{\psi} (y_2 -y_3)  \\
		 & \  \hspace{2cm} 
  		    K_{\psi}(y_1 -y_2)\sigma_{y_1}^{\psi}  \ {}^{\Psi}\mathfrak D_a^{\vec{\alpha}-\vec 1}[f](q,y_1)d\mu_{y_2}d\mu_{y_3}\cdots d\mu_{y_{n-1}} d\mu_{y_n}
		\\	
 +  & (-1)^{n} \displaystyle  \int_{J_{n}\times J_{n-1} \times \cdots \times J_2\times J_1} 
 \mathfrak K^{\vec{\alpha}}_{\psi, a} (y_n -q)  K_{\psi} (y_{n-1} -y_n) \cdots  K_{\psi} (y_2 -y_3) \\
 & \  \hspace{2cm} 
  K_{\psi} (y_1 -y_2) \ 
  {}^{\Psi}\mathfrak D_a^{\vec{\alpha}}[f](q,y_1) d\mu_{y_1} d\mu_{y_2}  \cdots 	d\mu_{y_{n-1}}d\mu_{y_n }  \\
		=  &      \left\{ \begin{array}{ll}  
		4f(q) + N[f](q,q, \vec{\alpha}) 
      		, &    x\in 
		J_n ,  \\ 0 , &  x\in \mathbb H\setminus\overline{J_n},                    
\end{array} \right. 
\end{align*} 

\item If ${}^{\Psi}\mathfrak D_a^{\vec{\alpha}}[f](q,x)=0$ for all $x\in J_1$ then  
\begin{align*}  &  \int_{\partial J_n}  \mathfrak K^{\vec{\alpha}}_{\psi, a} (y_n -x)\sigma_{y_n}^{\psi} {}^{\psi} \mathcal I_a^{y_{n}}[f](q,y_{n}, \vec{\alpha})   \\
   - &
\displaystyle \int_{J_n \times \partial J_{n-1}} \mathfrak K^{\vec{\alpha}}_{\psi, a} (y_n -x) K_{\psi}(y_{n-1} -y_n)\sigma_{y_{n-1}}^{\psi}  \
		{}^{\Psi}\mathfrak D_a^{\vec{\alpha}-(n-1)}[f](q,y_{n-1})d\mu_{y_n} 
  \\
	+ & \displaystyle  \int_{J_n  \times J_{n-1} \times \partial J_{n-2}}  \mathfrak K^{\vec{\alpha}}_{\psi, a} (y_n -x) K_{\psi} (y_{n-1} -y_n)
    K_{\psi}(y_{n-2} -y_{n-1})\sigma_{y_{n-2}}^{\psi} \\
					&  \  \hspace{3cm} 	{}^{\Psi}\mathfrak D_a^{\vec{\alpha}-(n-2)}[f](q,y_{n-2}) d\mu_{y_{n-1}} d\mu_{y_n}  \\ 
		 +			& \\
					 & \vdots					
					  \\
	+ & (-1)^{n-1}  \displaystyle  \int_{J_n  \times J_{n-1}\times \cdots  \times J_2 \times \partial J_1}  \mathfrak K^{\vec{\alpha}}_{\psi, a} (y_n -x) 
	K_{\psi} (y_{n-1} -y_n) \cdots  K_{\psi} (y_2 -y_3)  \\
		 & \  \hspace{2cm} 
  		 K_{\psi}(y_1 -y_2)\sigma_{y_1}^{\psi}  \ {}^{\Psi}\mathfrak D_a^{\vec{\alpha}-1}[f](q,y_1) d\mu_{y_2} d\mu_{y_3}\cdots d\mu_{y_{n-1}} d\mu_{y_n}
		\\
		=  &      \left\{ \begin{array}{ll}  
		\displaystyle \sum_{i=0}^3 f (q_0, \dots, x_i, \dots, q_3) + N[f](q,x, \vec{\alpha}) 
      		, &    x\in 
		J_n ,  \\ 0 , &  x\in \mathbb H\setminus\overline{J_n},                    
\end{array} \right. 
\end{align*}
		
\item If ${}^{\Psi}\mathfrak D_a^{\vec{\alpha}-k\vec 1}[f](q,x)=0$ for all $x \in \partial J_{k}$, for $k=1,\dots, n-1$, and   
${}^{\Psi}\mathfrak D_a^{\vec{\alpha}}[f](q,x) =0$ for all $x \in  J_{1}$  we obtain 
\begin{align*}  &  \int_{\partial J_n} \mathfrak K^{\vec{\alpha}}_{\psi, a} (y_n -x)\sigma_{y_n}^{\psi} {}^{\psi} \mathcal I_a^{y_{n}} [f](q,y_{n}, \vec{\alpha})   \\
		=  &  \left\{ 
	\begin{array}{ll}  
	\displaystyle \sum_{i=0}^3 f(q_0, \dots, x_i, \dots, q_3) + N[f](q,x, \vec{\alpha}) 
      		, &    x\in 
		J_n ,  \\ 0 , &  x\in \mathbb H\setminus\overline{J_n}.
\end{array} \right. 
\end{align*} 
If $q\in J_n$ doing $x=q$ we obtain 
\begin{align*}  \int_{\partial J_n} \mathfrak K^{\vec{\alpha}}_{\psi, a} (y_n -q)\sigma_{y_n}^{\psi} {}^{\psi} \mathcal I_a^{y_{n}} [f](q,y_{n}, \vec{\alpha})    
		=  4f(q) + N[f](q,q, \vec{\alpha}).
\end{align*} 
\end{enumerate}
\end{corollary}

\begin{remark}
Much of the formulas in previous proposition and corollaries can be extended to the right-hand versions of the iterated fractional $\Psi$-Fueter operator.
\end{remark}
\subsection*{Conclusion and Future development}
In this work, we extend basic results from $\psi-$hyperholomorphic function theory and the related operator calculus to the fractional setting. We consider fractional derivative in the sense of Riemann–Liouville to introduce a fractional form of the $\psi-$Fueter operator that depends on an additional vector of complex parameters with fractional real parts. More precisely, we present a fractional version of the classical quaternionic Borel-Pompeiu type formula associated to this fractional $\psi-$Fueter operator.

The results of Subsection \ref{3.2} have been encouraging enough to merit further investigation considering the general case $\Psi:= \{\psi_1 , \dots, \psi_n \}$ for  different structural sets, dealing with a generalized iterated fractional $\Psi$-Fueter operator of order $\vec{\alpha}$ on $AC^n(\overline{J_a^b}, \mathbb H)$ given by
\begin{align*} 
{}^{\Psi, \psi_{n+1}}\mathfrak D_a^{\vec{\alpha}}[f] (q,x)= {{}^{\Psi}\mathcal D_x } \circ {}^{\psi_{n+1}}  \mathcal I_a^x [f] (q,x, n\vec 1-\vec{\alpha}), 
 \end{align*}
for all $q,x\in J_a^b$    and similarly 
\begin{align*} 
{}^{\Psi, \psi_{n+1}}\mathfrak D_{r,a}^{\vec{\alpha}}[f](q,x) = {}^{\Psi}\mathcal D_{r,x}  \circ {}^{\psi_{n+1}}\mathcal I_a^x [f](q, x,n\vec 1-\vec{\alpha}), 
\end{align*}
where $\psi_{n+1}$ is another structural set. Work in this direction is currently under progress, some aspects are still challenging and require further research. 

\section*{Declarations}
\subsection*{Funding} Instituto Polit\'ecnico Nacional (grant number SIP20211188) and CONACYT.}
\subsection*{Conflict of interest} The authors declare that they have no conflict of interest regarding the publication of this paper.
\subsection*{Author contributions} Both authors contributed equally to the manuscript and typed, read, and approved the final form of the manuscript, which is the result of an intensive collaboration.

\end{document}